\documentclass[12pt]{amsart}
\usepackage{graphicx}
\usepackage{amssymb}
\usepackage{epstopdf}

\usepackage{amsmath}
\usepackage{amsthm}
\usepackage{a4wide}
\usepackage{pdfsync}
\usepackage{hyperref}
\usepackage{color}

\newtheorem{lemma}{Lemma}[section]
\newtheorem{theorem}{Theorem}[section]
\newtheorem{remark}{Remark}

\def\rr{\mathbb{R}}

\def\eps{\varepsilon}
\def\rhon{{\rho_n}}

\def\mf{\mathcal{F}}
\def\what{\widehat}

\def\la{\lambda}

\title[Nonlocal convection-diffusion]{ A compactness tool for the analysis of nonlocal evolution equations}
\author[L. I. Ignat ]{Liviu I. Ignat, Tatiana I. Ignat, Denisa Stancu-Dumitru}

\address{L. I. Ignat
\hfill\break\indent Institute of Mathematics ``Simion Stoilow'' of the Romanian Academy\\
\hfill\break\indent  21 Calea Grivitei Street \\010702 Bucharest \\ Romania 
\hfill\break\indent \and
\hfill\break\indent Faculty of Mathematics and Computer Science, University of Bucharest \\
\hfill\break\indent 14 Academiei Street, 010014 Bucharest, Romania.
}
 
 \email{{\tt
liviu.ignat@gmail.com}\hfill\break\indent  {\it Web page: }{\tt
http://www.imar.ro/\~\,lignat}}

\address{T. I. Ignat
\hfill\break\indent Institute of Mathematics ``Simion Stoilow'' of the Romanian Academy\\
\hfill\break\indent  21 Calea Grivitei Street \\010702 Bucharest \\ Romania 
}
 \email{\tt tatiana.ignat@gmail.com}

\address{Denisa Stancu-Dumitru
\hfill\break\indent Institute of Mathematics ``Simion Stoilow'' of the Romanian Academy\\
\hfill\break\indent  21 Calea Grivitei Street \\010702 Bucharest \\ Romania 
}
\email{\tt denisa.stancu@yahoo.com}

\begin{document}


\begin{abstract}
In this paper we give a new compactness criterion in the Lebesgue spaces $L^p((0,T)\times \Omega)$ and use it to obtain the first term in the asymptotic behaviour of the solutions of a nonlocal convection diffusion equation. We use previous results of Bourgain, Brezis and Mironescu  
to give a new criterion in the spirit of the Aubin-Lions-Simon Lemma.

\end{abstract}

\maketitle

\section{Introduction}

The aim of this paper is to give a new version of the classical compactness arguments in the  space $L^p((0,T)\times \Omega)$, \cite{MR916688}, one which can be adapted to nonlocal evolution equations.
We will apply this new criterion for the  analysis of  the long time behavior of the solutions of the following system 
\begin{equation}\label{cd}
\left\{
\begin{array}{ll}
u_t=J\ast u-u+G\ast |u|^{q-1}u -|u|^{q-1}u,& x\in \rr^d,t>0,\\[10pt]
u(0)=\varphi,
\end{array}
\right.
\end{equation}
where  $J$ and $G$ are non-negative $L^1(\rr)$-functions with mass one having a second and first momentum respectively. More details about the assumptions on these two functions will be given later. 

Let us now recall a classical compactness result  in the spaces $L^p((0,T),B)$, with $B$ a Banach space. Aubin-Lions-Simon Lemma \cite[Th.~5]{MR916688} assumes that we have three Banach spaces    $X\hookrightarrow B \hookrightarrow Y$ where the embedding  $X\hookrightarrow B$ is compact. A sequence $\{f_n\}_{n\geq 1}$ is relatively compact in
$L^p((0,T),B)$ (and in $C([0,T],B)$ if $p=\infty$) if we can guarantee that $\{f_n\}_{n\geq 1}$  is bounded in $L^p((0,T),X)$ and $\|\tau _h f_n-f_n\|_{L^p((0,T-h),Y)}\rightarrow 0$ as $h\rightarrow 0$ uniformly in $n$.

There are situations where we cannot bound uniformly a sequence $\{g_n\}_{n\geq 1}$ in a space that is compactly embedded in  $L^p(\Omega)$.  Instead of that we have estimates on some Dirichlet  forms that vary with $n$, estimates  that allow us to obtain the compactness of the sequence $\{g_n\}_{n\geq 1}$  (see for example \cite{1103.46310}, \cite{MR2041005} and \cite[Th.~6.11, p. ~128]{1214.45002}). Let us now be more precise. We choose
$1\leq p<\infty$ and $\Omega\subset \rr^d$  a smooth domain.  Function $\rho:\rr^d\rightarrow \rr$ is a nonnegative smooth  radial function with compact support, non identically zero, satisfying $\rho(x)\geq \rho(y)$ if $|x|\leq |y|$. Set
 $\rho_n(x)=n^d\rho(nx)$. 
Let $\{g_n\}_{n\geq 1}$ be a bounded sequence in $L^p(\Omega)$ such that
\begin{equation*}
n^p\int _{\Omega}\int _{\Omega} \rho_n(x-y) |g_n(x)-g_n(y)|^pdxdy\leq {M}.
\end{equation*}
Then as proved in  \cite{1103.46310}, \cite{MR2041005}, \cite[Th.~6.11, p. ~128]{1214.45002}, sequence  $\{g_n\}_{n\geq 1}$ is relatively compact in $L^p(\Omega)$.
Our main contribution is to use this compactness argument instead of the compact embedding $X\hookrightarrow B$ in the Aubin-Lions-Simon Lemma and to obtain a new compactness criterion in $L^p((0,T)\times \Omega)$. 

The main compactness tool that we prove in this paper is the following one. 
\begin{theorem}\label{rossi-time}
Let $1\leq p<\infty$ and $\Omega\subset \rr^d$ be an open set. Let $\rho:\rr^d\rightarrow \rr$ be a nonnegative smooth  radial function with compact support, non identically zero, and $\rho_n(x)=n^d\rho(nx)$. Let $\{f_n\}_{n\geq 1}$ be a sequence of functions  in $L^p((0,T)\times \Omega)$ such that
\begin{equation}\label{hyp.1}
\int _0^T \int _{\Omega} |f_n(t,x)|^pdxdt \leq \ {M}
\end{equation}
and
\begin{equation}\label{hyp.2}
n^p\int _0^T\int _{\Omega}\int _{\Omega} \rho_n(x-y) |f_n(t,x)-f_n(t,y)|^pdxdydt\leq {M}.
\end{equation}

1. If  $\{f_n\}_{n\geq 1}$ is weakly convergent in $L^p((0,T)\times \Omega)$ to $f$ then $f\in L^p((0,T),W^{1,p}(\Omega))$ for $p>1$ and $f\in L^1((0,T),BV(\Omega))$ for $p=1$.

2. Let $p>1$. Assuming  that $\Omega$ is a smooth bounded domain in $\rr^d$, $\rho(x)\geq \rho(y)$ if $|x|\leq |y|$ and that
\begin{equation}\label{hyp.3}
\|\partial_t f_n\|_{L^p((0,T),W^{-1,p}(\Omega))}\leq M
\end{equation}
then  $\{f_n\}_{n\geq 1}$ is relatively compact in $L^p((0,T)\times \Omega)$.
\end{theorem}

Extensions to mixed type space norms of the type $L^p((0,T),L^q(\Omega))$ could also be  obtained by adapting the estimates in this paper. The possibility of obtaining  more general nonlocal compactness tools as in Aubin-Lions-Simon Lemma (see Theorem \ref{3spaces} below) remains to be analyzed. 
In \eqref{hyp.3} for technical reasons we considered the space $W^{-1,p}(\Omega)$ but we believe that the results still hold by replacing $W^{-1,p}(\Omega)$ with any space $Y$ such that $L^p(\Omega) \hookrightarrow Y$ continuously. More general functions $\rho_n$ can be considered, like  those in \cite{MR2041005} but this is beyond the scope of this article. We leave to the reader the possible extension of the above result to the case when $\rho_n$ are not necessarily obtained by rescaling a given function $\rho$. The case of general weights $\rho_n$ will introduce new difficulties since technical results in 
Lemma \ref{localizare.n} and Lemma \ref{balance-p}, used  in the proof of Theorem \ref{rossi-time},
make use in an essential manner of the fact that the weights are obtained by scaling from a given function $\rho$.

Once we prove Theorem \ref{rossi-time} we apply it in the analysis of the asymptotic behaviour of system \eqref{cd}. Recently, similar results to those in Theorem \ref{rossi-time} have been employed in the analysis of some numerical splitting methods for Burgers like equations \cite{pozo-ignat}.

Let us now be more precise about the assumptions on the kernels $J$ and $G$. We assume that $J,G:\rr^d\rightarrow \rr$ are non-negative functions satisfying
the following assumptions 

(H0) $\displaystyle \int _{\rr^d}J(z)dz=\int _{\rr^d}G(z)dz=1$,\\[3pt]

(H1) $J\in L^1(\rr^d,1+|x|^2), G\in L^1(\rr^d,1+|x|),$

(H2) $J$ is positive in a neighborhood of the origin \\[2pt]

(H3) $J$ is symmetric, i.e. $J(z)=J(-z)$. In particular the first moment of $J$ vanishes
$$\int _{\rr^d}J(z)z_jdz=0, \quad j=1,\dots, d.$$

Condition (H0) is assumed for simplicity. In fact the same analysis can be done for integral equations of the type
\begin{equation*}\label{cd-gen}
\left\{
\begin{array}{ll}
u_t(t,x)=&\displaystyle \int _{\rr^d}J(x-y)( u(t,y)-u(t,x))dy+\\[10pt]
&\quad \displaystyle+\int _{\rr^d}G(x-y)(|u|^{q-1}u(t,y) -|u|^{q-1}u(t,x))dy, \quad  x\in \rr^d, t>0,\\[10pt]
u(0)=\varphi,
\end{array}
\right.
\end{equation*}
without assuming that the mass of each of the functions $J$ and $G$ is one.

The well-posedness of this model  has been analyzed in \cite[Th.~1.1]{MR2356418} under smoothness assumptions of the functions $J$ and $G$. We emphasize that the same global well-posedness result holds under the assumptions that $J$ and $G$ are non-negative $L^1(\rr^d)$-functions:
for any $q>1$ and $\varphi\in L^1(\rr^d)\cap L^\infty(\rr^d)$ there exists a unique global solution $u\in C([0,\infty), L^1(\rr^d)\cap L^\infty(\rr^d))$ satisfying 
$$\|u(t)\|_{L^1(\rr^d)}\leq \|\varphi\|_{L^1(\rr^d)}\quad \text{and} \quad   
\|u(t)\|_{L^\infty(\rr^d)}\leq \|\varphi\|_{L^\infty(\rr^d)}.$$
Since $J$ and $G$ have mass one the mass conservation property holds
$$\int_{\rr^d}u(t,x)dx=\int_{\rr^d}\varphi(x)dx.$$
Since the proof of the global well-posedness follows by the same fix point argument as in \cite{MR2356418} we will omit it here.

Observe that under the assumptions (H0), (H1) and (H3) there exist positive constants $R$, $\delta$
such that 
\[
\left\{
\begin{array}{cc}
 \widehat J(\xi)\leq 1-\frac{A}2 |\xi|^2, & |\xi|\leq R,     \\[10pt]
  \widehat J(\xi)\leq 1-\delta, & |\xi|\geq R,     
\end{array}
\right.
\]
where $A=\frac 12 \int _{\rr^d}J(z)|z|^2dz$.
Under these assumptions  it has been proved 
 in \cite[Th.~1.4]{MR2356418} that the solutions of \eqref{cd} decay similar to the classical heat equation: for any $1\leq p<\infty$ the following holds:
\begin{equation}\label{decay}
\|u(t)\|_{L^p(\rr^d)}\leq C(p,d,\|\varphi\|_{L^1(\rr^d)}, \|\varphi \|_{L^\infty(\rr^d)}) (t+1)^{-\frac d2(1-\frac 1p)}.
\end{equation}
This decay property has been obtained in \cite{MR2356418} by the so-called \textit{Fourier Splitting method} \cite{MR571048, MR856876, MR1356749} and in a more general setting in \cite{MR2542582}. When $p=2$ a similar argument has also been  used in \cite{MR2103702}. As far as the authors know, 
the case $p=\infty$ in \eqref{decay} is  open.

In the case when the nonlinear term is supercritical, i.e. $q>1+1/d$, the first term in the asymptotic behavior has been analyzed in \cite{MR2356418} under the additional assumption that $J\in \mathcal{S}(\rr^d)$, the class of rapidly decreasing functions. There the main idea was that the nonlinear part decays faster than the linear semigroup and then the first term in the long time behavior is given by the linear semigroup. 
This has been already observed in \cite{0762.35011} in the case of the classical convection-diffusion system.

The aim of this paper is to give an answer to the critical case $q=1+1/d$ even though we give a proof that both treats the critical and super-critical case. We emphasize that we do not require  function $J$ to  belong to $\mathcal{S}(\rr^d)$ as in \cite{MR2356418}.
The subcritical case $1<q<1+1/d$ is still open. The method we employ is  the so-called \textit{four step method} that consists in the analysis of some rescaled
orbits $\{u_\lambda(t)\}$. We refer to \cite{MR1977429} for a review of the method in the case of the porous medium equation.

We consider two important quantities
$$A=\frac 12 \int _{\rr^d}J(z)|z|^2dz\quad \text{and}\quad B=(B_1,\dots, B_d),  B_j=\int _{\rr^d}G(z)z_jdz, j=1,\dots, d.$$
The main result concerning system \eqref{cd} is the following one. 

\begin{theorem}\label{asimp}
Let $J$ and $G$ be two non-negative functions satisfying hypotheses (H0)-(H3).
For any $\varphi\in L^1(\rr^d)\cap L^\infty(\rr^d)$ the solution $u$ of system \eqref{cd} satisfies
\begin{equation}\label{lim.t}
\lim _{t\rightarrow \infty} t^{\frac  d2(1-\frac 1p)}\Big\|u(t) -t^{-d/2} f_m\Big(\frac {x}{\sqrt t}\Big)\Big\|_{L^p(\rr^d)}=0, \quad 1\leq p<\infty,
\end{equation}
where the profile $f_m$ is the smooth solution of  the equation
$$-A\Delta f_m-\frac 12 x\cdot \nabla f_m=\frac d2 f_m - \alpha B\cdot \nabla (|f_m|^{q-1}f_m) \quad \text{in}\  \rr^d,$$
 with $\int_{\rr^d} f_m=m$ where $m$ is the mass of the initial data $\varphi$ and 
 $$\alpha=
 \left\{
 \begin{array}{ll}
 1, & q=1+\frac 1d,\\[10pt]
 0, & q>1+\frac 1d.
 \end{array}
 \right.
 $$
\end{theorem}

Next, we say a few words about the above asymptotic profile  $$U(t,x)=t^{-d/2} f_m\Big(\frac {x}{\sqrt t}\Big).$$
 When $q>1+1/d$ or $B=0_{1,d}$ the asymptotic profile   is  the rescaled heat kernel solution of 
 \begin{equation}\label{heat}
\left\{
\begin{array}{ll}
U_t=A \Delta U,& x\in \rr^d, t>0,\\[10pt]
U(0)=m\delta _0.
\end{array}
\right.
\end{equation} 
 When $q=1+1/d$ and $B\neq 0_{1,d}$, $U$ is the unique solution of the following equation
\begin{equation}\label{burgers}
\left\{
\begin{array}{ll}
U_t=A \Delta U- B\cdot \nabla (|U|^{1/d}U),& x\in \rr^d, t>0,\\[10pt]
U(0)=m\delta _0.
\end{array}
\right.
\end{equation}
The well-posedness of this system has been analyzed in \cite{MR1233647} in the one-dimensional case and  in \cite{MR1266100} in the multi-dimensional case.
It has been proved in \cite{MR1032626} that the profile $f_m$ is of constant sign and decays exponentially to zero as $|x|\rightarrow \infty$. 

We remark that in the case of the symmetric function $G$, i.e. $G(z)=G(-z)$, the solution of \eqref{cd} converges to the heat kernel since in this case $B$ vanishes. When $B\neq 0$ we obtain in the limit the solutions of the viscous convection-diffusion equation. Along the paper we will consider the case of nonnegative initial data, so nonnegative solutions of system \eqref{cd}. The case of sign-change solutions could also be analyzed with small modifications of the proof (see \cite{MR1233647} for a rigorous treatment of the critical case for the convection-diffusion equation). 

In the linear case, i.e. $u_t=J\ast u-u$, the asymptotic behavior has been obtained in \cite{MR2257732} by means of Fourier analysis techniques and in \cite{MR3155537} by scaling methods. In \cite{MR3155537} the scaling argument works since it is applied to the smooth part of the solution.  
Refined asymptotics  have been obtained in \cite{MR2460931, MR2557678}. 
We also recall here \cite{MR2748435, MR2805821} where a scaling method is used for equations of the type $u_t=J\ast u-u-u^p$. There the authors obtain barriers for  $W$ and its derivatives, $W$ being the smooth part  of the solution of the linear equation $u_t=J\ast u-u$. Once these barriers are obtained the authors split  the solution of the nonlinear problem in  a way that permits obtaining uniform H\"older estimates and then compactness.
The method developed here is more flexible in the sense that it uses only  energy estimates that involve the linear part of the equation and the good sign of the nonlinearity.

In the local case, i.e. $u_t=\Delta u+a\cdot \nabla (|u|^{q-1}u)$, the same analysis has been performed in a series of papers. In \cite{0762.35011} the case $q\geq 1+1/d$ is treated and  the results in the  critical case have been  obtained by a careful space-time change of variables and using weighted Sobolev spaces. The sub-critical case is more difficult and  the one-dimensional case has been considered in \cite{MR1233647}. The extension to higher dimensions has been obtained in \cite{MR1266100} and \cite{MR1440033}. 

In contrast with the analysis in \cite{0762.35011} here we assume that the initial data belong to the space $L^1(\rr^d)\cap L^\infty(\rr^d)$. This assumption is necessary since even in the linear case $u_t=J\ast u-u$ a lack of smoothing effect is present. More precisely the solutions of the linear model are as regular as the initial data. In the case of the heat equation with initial data in $L^1(\rr^d)$ the solution at any positive time belongs to any $L^p(\rr^d)$ space $1\leq p\leq \infty$,   and this type of gain of integrability can  also be proved for the nonlinear convection-diffusion \cite{0762.35011}. 

We recall some  similar models to those analyzed here. In \cite{MR2138795} the author considers a one-dimensional model that is nonlocal in the diffusive part $u_t=J\ast u-u+uu_x$ with $J=e^{-|x|}$ and he proves that its solutions converge to the ones of Bourger's equation with Dirac delta initial data. However the key tool used there, an Oleinik estimate $\partial_x u(t)\leq 1/t$ in $\mathcal{D}'(\rr)$  is not available in our model. The methods used here can be adapted to analyze similar models but with nonlinearities of the type $(u^q)_x$, $q\geq 2$,   \cite{MR3190994}. 
In these cases, entropy conditions in the sense of Kru{\v{z}}kov \cite{MR0267257}  should be imposed on weak solutions in order to have a well-posed problem.
This does not appear in our model since the nonlinearity does not involve derivatives.

The models considered here could be related with the ones considered in \cite{MR2924787} where a scalar conservation law with a diffusion-type source of the type $u_t+\nabla \cdot f(u)=\Delta P_s u$ is analyzed. There $P_s$ is essentially given by $\widehat {P_su}(\xi)\simeq (1+|\xi|^2)^{-s}\hat u(\xi)$ and even more general models are considered. However, in order to obtain the long time behavior of the solutions, the authors assume that the initial data belong to some $H^N(\rr^d)$ spaces where $N$ is large enough. In view of our analysis here we believe that the results in \cite{MR2924787} can be obtained by assuming less regularity on the initial data.
The analysis of these models by our methods remains to be considered in future papers.

Similar  nonlocal models have been introduced recently in \cite{MR2888353} where  the nonlocal convective term takes the form
$$\int _{\rr}\phi_0(y-x)\Big(\frac{u(t,y)+u(t,x)}2\Big)^2dy$$
where $\phi_0$ is an odd  function. The possible application of the methods introduced here remains to be analyzed. From the very beginning, the main difficulty in these models is the global existence of the solutions.
Some models when the convection is nonlocal have been considered previously in \cite{tesis-tatiana}, $u_t=u_{xx}+G\ast u^q-u^q$, $q\geq 2$ and \cite{MR2739418} , $u_t=u_{xx}+(u^{q-1}(K\ast u))_x$, $q=2$. The main difficulty in obtaining the asymptotic behavior for similar models where the convection is dominant, i.e. $1<q<2$, is to  obtain an entropy estimate. Even if the entropy inequality can be avoided in the critical case  it seems to be crucial for the uniqueness of the solutions of the limit equation in the sub-critical case.  However, we refer to \cite{MR1440033} where
the asymptotic behavior of systems of the type $u_t=\Delta u-\partial _y(|u|^{q-1}u)$ with $q$ subcritical
is obtained without entropy estimates  but rather with a kinetic formulation that allows to use some compactness arguments
  previously employed in the case of multidimensional scalar conservation laws \cite{MR1201239}. The possible application of these kinetic methods to the case on nonlocal diffusion and/or convection remains to be analyzed in the future.

The paper is organized as follows. In Section \ref{prelim} we review a few compactness arguments known to be useful in the analysis of time evolution problems and prove Theorem \ref{rossi-time}.
 Once the compactness tool is obtained,  in Section \ref{proof.main.result} we  prove Theorem \ref{asimp}.

\section{Compactness Tools}\label{prelim}\setcounter{equation}{0}

In this section we review a few classical compactness tools and give some results that will allow us to prove the main result of this paper.

First we recall some results given in \cite{MR916688} about the characterization of compact sets in $L^p((0,T),B)$ where $B$ is a Banach space and $1\leq p\leq \infty$.

\begin{theorem}[\cite{MR916688}, Th. 1] \label{simon}
Let $\mf \subset L^p((0,T),B)$. $\mf$ is relatively compact in $L^p((0,T),B)$ for $1\leq p<\infty$, or $C([0,T],B)$ for $p=\infty$ if and only if
\begin{enumerate}
\item \label{condition.1} 
 $\displaystyle \Big\{\int _{t_1}^{t_2} f(t)dt, \, f\in \mf\Big\}$ is  relatively compact in $B$ for all $0<t_1<t_2<T$,\\

\item  \label{condition.2}  $\|\tau _h f-f\|_{L^p((0,T-h),B)}\rightarrow 0$ as $h\rightarrow 0$ uniformly for $f\in \mf$. 
\end{enumerate}
\end{theorem}
The following  criterion  is also given.
\begin{theorem}[\cite{MR916688}, Th.~5] \label{3spaces}
Let us consider three Banach spaces  $X\hookrightarrow B \hookrightarrow Y$ where $X\hookrightarrow B$ is compact.
Assume $1\leq p\leq \infty$ and \\[5pt]
i) $\mf$ is bounded in $L^p((0,T),X)$,\\[5pt]
ii) $\|\tau _h f-f\|_{L^p((0,T-h),Y)}\rightarrow 0$ as $h\rightarrow 0$ uniformly for $f\in \mf$.

Then $\mf $ is relatively compact in $L^p((0,T),B)$ (and in $C([0,T],B)$ if $p=\infty$).
\end{theorem}

The last criterion is obtained mainly by using Theorem \ref{simon} and the following inequality that follows from the fact that $X$ is compactly embedded in $B$: for any $\eps>0$ there exists $\eta(\eps)>0$ such that 
\begin{equation}\label{ineg.eps}
\|u\|_B\leq \eps \|u\|_X+\eta(\eps) \|u\|_Y,\quad \forall \, u \in X.
\end{equation}
In the nonlocal setting we will obtain  a similar inequality  in Lemma \ref{localizare.n}.

\medskip
Now we recall  some compactness results that have been proved in the nonlocal context \cite{1103.46310}, \cite{1214.45002} and in a more general setting in \cite{MR2041005}.

\begin{theorem}[\cite{1214.45002}, Th.~6.11, p. ~128]\label{rossi}
Let $1\leq p<\infty$ and $\Omega\subset \rr^d$ be an  open set. Let $\rho:\rr^d\rightarrow \rr^d$ be a nonnegative smooth  radial function with compact support, non identically zero, and $\rho_n(x)=n^d\rho(nx)$. Let $\{f_n\}_{n\geq 1}$ be a bounded sequence in $L^p(\Omega)$ such that
\begin{equation}\label{ros.1}
n^p\int _{\Omega}\int _{\Omega} \rho_n(x-y) |f_n(x)-f_n(y)|^pdxdy\leq {M}.
\end{equation}
The following hold:

1. If $\{f_n\}_{n\geq 1}$ is weakly convergent in $L^p(\Omega)$ to $f$ then $f\in W^{1,p}(\Omega)$ for $p>1$ and $f\in BV(\Omega)$ for $p=1$. Moreover 
$$\|\nabla f\|_{L^{p}(\Omega)}\leq C(\Omega,\rho) M.$$

2. Assuming  that $\Omega$ is a smooth bounded domain in $\rr^d$ and $\rho(x)\geq \rho(y)$ if $|x|\leq |y|$ then $\{f_n\}_{n\geq 1}$ is relatively compact in $L^p(\Omega)$.
\end{theorem}

We point out that the  assumption on the compact support of function $\rho$ could be removed. In fact once we have estimate \eqref{ros.1} for a function $\rho$ we also have this estimate for any  compactly supported function $\tilde \rho$ with $\tilde \rho\leq \rho$.


The above results hold under more general assumptions on the weights $\{\rho_n\}_{n\geq 1}$ and on a bounded domain $\Omega$  in $\rr^d$ with Lipshitz boundary.
As proved in \cite[Th.~1.2]{MR2041005} we can assume that  $\{\tilde\rho_n\}_{n\geq 1}$ is a sequence of radially symmetric functions in $L^1(\rr^d)$ satisfying
\begin{equation}\label{cond.rho}
\left\{
\begin{array}{ll}
\tilde \rho_n\geq 0,\ \text{a.e. in }\ \rr^d,\\[10pt]
\int_{\rr^d}\tilde\rho_n(x)=1, \, \forall \ n\geq 1,\\[10pt]
\lim _{n\rightarrow \infty} \int _{|x|>\delta}\tilde \rho_n(x)dx=0, \quad \forall \delta>0
\end{array}
\right.
\end{equation}
and that
$$\int _\Omega \int _\Omega \frac{\tilde \rho_n(x-y)}{|x-y|^p} |f_n(x)-f_n(y)|^pdxdy\leq {M}.$$
Then the results in Theorem \ref{rossi} remain true in dimension $d\geq 2$. In dimension $d=1$ some technical assumptions have to be assumed \cite[Th.~1.3]{MR2041005}. 
Choosing $\tilde \rho_n(x)= n^d (n|x|)^p \rho(nx)$ with $\rho$ radial and decreasing, these technical assumptions hold and we obtain the results in the second part of Theorem \ref{rossi}.
We also recall that under the above conditions on $\tilde \rho_n$ a Poincar\'e  inequality holds \cite[Th.~1.1]{MR2041005}
$$\Big\|f_n- \frac 1{|\Omega|}\int _{\Omega} f_n\Big\|^p_{L^p(\Omega)}\leq C(p,\Omega,\{\rho_n\})  \int _{\Omega}\int _{\Omega} \frac{\tilde\rho_n(x-y)}{|x-y|^p} |f_n(x)-f_n(y)|^pdxdy.$$
In view of this inequality the boundedness of  $\{f_n\}_{n\geq 1}$ in $L^p(\Omega)$ is guaranteed by \eqref{ros.1}  if we assume that   $\{f_n\}_{n\geq 1}$ is bounded in  $L^1(\Omega)$ and $\Omega$ has  finite measure.

%
%
%
%

\begin{proof}[Proof of Theorem \ref{rossi-time}]
Using the same arguments as in the proof of Theorem \ref{rossi} (see \cite[Ch.~6, p.~128]{1214.45002}) we obtain the results in the first part.


We now prove the second part of the theorem by
following the ideas in \cite{MR916688} but taking into account the particular estimate  \eqref{hyp.2}. From now on, in order to simplify the presentation,  we assume that $\rho$ is  supported in the unit ball.

{\bf Step. I. Compactness in $L^p((0,T),W^{-1,p}(\Omega))$.} We now check the hypotheses in Theorem \ref{simon}. Let us choose $0\leq t_1<t_2\leq T$ and set
$$g_n(x)=\int _{t_1}^{t_2} f_n(s,x)ds. $$
Estimate \eqref{hyp.2} gives us that 
\[
n^p \int _{\Omega}\int _{\Omega} \rho_n(x-y)|g_n(x)-g_n(y)|^pdxdy\leq MT^{p-1}.
\]
Theorem \ref{rossi} applied to sequence $\{g_n\}_{n\geq 1}$ 
shows that there exists $g\in W^{1,p}(\Omega)$ such that, up to a subsequence, $g_n\rightarrow g$ in $L^p(\Omega)$ so in $W^{-1,p}(\Omega)$.
Estimate \eqref{hyp.3} shows that the second requirement in Theorem \ref{simon} is also satisfied. Hence $\{f_n\}_{n\geq 1}$ is relatively compact in
$L^p((0,T),W^{-1,p}(\Omega))$.  

\medskip

{\bf Step. II. Compactness in $L^p((0,T),L_{loc}^p(\Omega))$.} 
Since $\{f_n\}_{n\geq 1}$ is bounded in $L^p((0,T)\times \Omega)$ then up to a subsequence $\{f_n\}_{n\geq 1} $ weakly converges to some function $f$ in $L^p((0,T)\times \Omega)$. The first part of Theorem \ref{rossi-time}   guarantees that $f\in L^p((0,T),W^{1,p}(\Omega))$.

We now use the strong convergence  in $L^p((0,T),W^{-1,p}(\Omega))$ obtained in Step I,  estimate \eqref{hyp.2} and the fact that $f\in L^p((0,T),W^{1,p}(\Omega))$ to prove that up to a subsequence, $\{f_n\}_{n\geq 1}$ strongly convergences to $f$ in $L^p((0,T)\times \Omega)$. 

To simplify the presentation we will always denote the subsequence by $\{f_n\}_{n\geq 1}$. Also, when possible, we will not write all the constants in inequalities of the type $f\leq Cg$ using  instead  $f\lesssim g$.

In the following we prove that for any   $\Omega'\subset \Omega$ such that $d( \Omega',\partial \Omega)>0$, $\{f_n\}_{n\geq 1}$ is relatively compact in $L^p((0,T)\times \Omega')$. 
From now on for a set $\mathcal{O}$ we will denote 
$$\mathcal{O}_{\rho}=\mathcal{O}+{\rm supp} (\rho)=\{x+\theta, \ x\in \mathcal{O}, \theta \in {\rm supp}(\rho)\}.$$

The following two Lemmas will be very useful in our analysis. Their proof will be given later.

\begin{lemma}\label{est.ariba}Let $\Omega$ be and open set of $\rr^d$.
For any $1<p<\infty $ there exists a positive constant $C(\rho,\Omega,p)$ such that the following inequality 
\begin{equation*}
n ^p\int _{\Omega} \int _{\Omega} \rho_n (x-y)|u(x)-u(y)|^pdxdy\leq C(\rho,\Omega,p) \int _{\Omega} |\nabla u|^p
\end{equation*}
holds for all $n >0$ and  $u\in W^{1,p}(\Omega)$.
\end{lemma}

\begin{lemma}\label{localizare.n}
Let $\Omega$ be a bounded domain and $\chi\in C_c^1(\Omega)$. There exists a positive constant $C=C(\Omega, \chi, \rho,p)$ such that for every $\eps\in (0,1)$ the following inequality
\begin{equation}\label{balance.n}
C\int _{\Omega} |\chi u|^p\leq  \eps  n^p \int_{\Omega_\rhon}\int _{\Omega_\rhon} \rho_n(x-y)|u(x)-u(y)|^pdxdy +  \eps \int _{\Omega_\rhon}|u|^p +\frac 1\eps \|u\|_{W^{-1,p}(\Omega)}^p
\end{equation}
holds for all $n\eps^{1/p}\gtrsim 1$ and for all $u\in L^p(\rr^d)$. 
\end{lemma}

\begin{remark}
In the right hand side of inequality \eqref{balance.n} we have $\eps^{-1}$ and the $W^{-1,p}(\Omega)$-norm.  We believe that  some improvement in \eqref{balance.n} can be done  by allowing the norm of the last term 
 to be in some space $Y$ with $L^p(\Omega)\hookrightarrow Y$ and replacing $\eps^{-1}$ correspondingly. The extension of  Lemma \ref{localizare.n} to general spaces $Y$ will enlarge the class of nonlocal problems where the scaling arguments used in this paper can be applied. 
\end{remark}

Let us fix $\Omega'\subset \Omega$ such that $d(\Omega',\partial \Omega)>0$ and  choose a smooth function $\chi$ compactly supported in $\Omega$ such that $\chi \equiv 1$ in $\Omega'$.
We choose $N_0$ large enough such that 
$\Omega_\rhon'\subset \Omega$ for all $n\geq N_0$.

Applying  Lemma \ref{localizare.n}  with $g=f_n-f$   to the set $\Omega'$
we have for any $n\gtrsim \eps^{-1/p}$  that 
\begin{equation}\label{est.g}
\|\chi g\|_{L^p(\Omega')}^p\lesssim \eps n^p \int _{\Omega_\rhon'}\int _{\Omega_\rhon'}\rho_n(x-y)| g(x)- g(y)|^pdxdy +\eps \int _{\Omega_\rhon'} |g|^p+\frac 1\eps \|g\|_{W^{-1,p}(\Omega')}^p.
\end{equation}
We integrate the above inequality on   the time interval $[0,T]$ and use that
$\|g\|_{W^{-1,p}(\Omega')}\leq \|g\|_{W^{-1,p}(\Omega)}$
to obtain 
\begin{align*}
\int _0^T\int _{\Omega'} \chi^p |f _n-f|^p\lesssim &\eps n^p \int _0^T\int _{\Omega_\rhon'}\int _{\Omega_\rhon'} \rho_n(x-y)| (f_n-f)(t,x)- (f_n-f)(t,y) |^pdxdydt\\
&+\eps \int _0^T\int _{\Omega_\rhon'} |f_n-f|^p
+\frac {1}\eps  \int _0^T \|f_n-f\|_{W^{-1,p}(\Omega)}^p\\
\lesssim & \eps n^p \int _0^T\int _{\Omega_\rhon'}\int _{\Omega_\rhon'} \rho_n(x-y)|   f_n(t,x)- f_n(t,y)|^pdxdydt\\
& +\eps n^p \int _0^T\int _{\Omega_\rhon'}\int _{\Omega_\rhon'}  \rho_n(x-y) | f(t,x)-f(t,y) |^pdxdydt\\
&+\eps \Big(\|f_n\|_{L^p((0,T)\times \Omega'_{\rho_n})}^p+\|f\|_{L^p((0,T)\times \Omega'_{\rho_n})}^p\Big)+\frac {1}\eps  \int _0^T \|f_n(t)-f(t)\|_{W^{-1,p}(\Omega)}^p.
\end{align*}
Since for $n\gtrsim \max\{N_0,\eps^{-1/p}\}$ we have that $\Omega_\rhon'\subset \Omega$ 
we use estimates \eqref{hyp.1}, \eqref{hyp.2}, Lemma \ref{est.ariba} and the fact that $f\in L^p((0,T),W^{1,p} (\Omega))$ to obtain that
$$\int _0^T\int _{\Omega'} \chi^p |f _n-f|^p\lesssim
 \eps (M+\|f \|^p_{L^p((0,T),W^{1,p}(\Omega))})+\frac {1}\eps  \int _0^T \|f_n(t)-f(t)\|_{W^{-1,p}(\Omega)}^pdt.$$
Using Step I, up to a subsequence,  we obtain that for any $\eps\in (0,1)$
$$\limsup _{n\rightarrow \infty} \int _0^T\int _{\Omega'}  |f _n-f|^pdxdt \lesssim \eps (M+\|f \|^p_{L^p((0,T),W^{1,p}(\Omega))}).$$
Then  $f_n$ strongly converges to $f$ in $L^p((0,T)\times \Omega')$.
Applying a standard diagonalisation procedure we can extract a subsequence, denoted again by $\{f_n\}_{n\geq 1}$, such that $f_{n}\rightarrow f$ in $L^p((0,T),L^p_{loc}(\Omega))$.


\medskip

{\bf Step. III. Compactness in $L^p((0,T),L^p(\Omega))$.} We now use the following result in \cite[Lemma~5.1, Lemma~7.2]{MR2041005}.
For a positive number $r>0$ we set
$$\Omega_r:=\{x\in \Omega: d(x,\partial \Omega)>r\}.$$
\begin{lemma}\label{ponce}
Let $\Omega$ be a bounded Lipschitz domain of $\rr^d$.
There exist constants $r_0>0$ depending on $\Omega$ and on $\rho$ and $C_1,C_2$ (depending on $p,\Omega$ and $d$) so that the following holds: given $0<r<r_0$ we can find $N_0\geq 1$ such that
 \begin{equation}\label{est.ponce}
\int _{\Omega} |g|^p\leq C_1 \int _{\Omega_r}|g|^p+C_2 r^p n^p \int _{\Omega}\int _{\Omega} \rho_n(x-y)|g(x)-g(y)|^pdxdy
\end{equation}
for every $g\in L^p(\Omega)$ and $n\geq N_0$.
\end{lemma}

We apply the above Lemma with  $g=f_n-f$ and integrate the resulted inequality on the time interval $(0,T)$. Thus
\begin{align*}
\int _0^T \int _{\Omega}& |f_n-f|^p \\
& \lesssim \int _0^T \int _{\Omega_r}|f_n-f|^p+ r^pn^p \int _0^T \int _{\Omega}\int _{\Omega} \rho_n(x-y)|(f_n-f)(t,x)-(f_n-f)(t,y)|^pdxdydt\\
&\lesssim  \int _0^T \int _{\Omega_r}|f_n-f|^p+ r^pn^p\int _0^T  \int _{\Omega}\int _{\Omega} \rho_n(x-y)|f_n(t,x)-f_n(t,y)|^pdxdydt\\
&\quad + r^pn^p\int _0^T  \int _{\Omega}\int _{\Omega} \rho_n(x-y)|f(t,x)-f(t,y)|^pdxdydt.
\end{align*}
Using estimate \eqref{hyp.2} and Lemma \ref{est.ariba} we get
$$
\int _0^T \int _{\Omega} |f_n-f|^p \lesssim  \int _0^T \int _{\Omega_r}|f_n-f|^p+ r^p M+ r^p \int _0^T \int _{\Omega}|\nabla f|^p.
$$
Since $ f_{n}\rightarrow f$ in $L^p((0,T),L^p_{loc}(\Omega))$ we can let $n\rightarrow \infty$  and then for any $r\in (0,r_0)$ we have 
$$\limsup_{n\rightarrow\infty} \int _0^T \int _\Omega |f_n-f|^p\lesssim r^p \Big(M+\int _0^T \int _{\Omega}|\nabla f|^p\Big ).$$
This implies that, up to a subsequence, $f_n\rightarrow f$ in $L^p((0,T),L^p(\Omega))$ and the  proof of Theorem  \ref{rossi-time} is now finished.
\end{proof}

\begin{proof}[Proof of Lemma \ref{est.ariba}] We first consider the case when  $\Omega=\rr^d$. 
By scaling, it is sufficient to consider the case $n=1$. Since
\begin{equation}\label{taylor-1}
u(x)-u(y)=\int _0^1 (x-y)\cdot \nabla u (y+s(x-y))ds
\end{equation}
we get that
\begin{align*}
\iint _{\rr^{2d}} \rho(x-y)|u(x)-u(y)|^pdxdy&\leq \iint _{\rr^{2d}} \rho(x-y)|x-y|^p  \int _0^1 |\nabla u (y+s(x-y))|^pdsdxdy\\
&= \int _{\rr^d}\rho(z)|z|^p\int _{\rr^d}|\nabla u|^p.
\end{align*}

In the case of a bounded domain $\Omega$ we first extend $u$ to $\rr^d$ such that $\|\nabla u\|_{L^p(\rr^d)}\leq C(\Omega)\|\nabla u\|_{L^p(\Omega)}$.
Then we have
\begin{align*}
n ^p\iint _{\Omega\times \Omega} & \rho_n (x-y)|u(x)-u(y)|^pdxdy\leq 
n ^p\iint _{\rr^{2d}} \rho_n (x-y)|u(x)-u(y)|^pdxdy\\
&\leq C(\rho,p) \int _{\rr^d}  |\nabla u|^p\leq C(\rho,\Omega,p) \int _{\Omega} |\nabla u|^p.
\end{align*}
The proof of Lemma \ref{est.ariba} is now complete.
\end{proof}

The rest of this  subsection is devoted to the proof of Lemma \ref{localizare.n}. 
In order to give its proof we need some auxiliary Lemmas.


%
%
%
%
%
%
%

\begin{lemma}\label{balance-p}
Let $1<p<\infty$. There exists a positive constant $C=C(\rho,p,d)$ such that 
for every $\eps\in (0,1)$ the following inequality
\begin{equation}\label{est.balance}
C\|u\|_{L^p(\rr^d)}^p\leq \eps \Big[n^p \int_{\rr^d}\int _{\rr^d} \rho_n(x-y)|u(x)-u(y)|^pdxdy + \|u\|_{L^p(\rr^d)}^p\Big]+ \eps^{-1} \|u\|_{W^{-1,p}(\rr^d)}^p
\end{equation}
holds for all $n\eps^{1/p}\gtrsim  1$ and for all $u\in L^p(\rr^d)$.
\end{lemma}

Before starting the proof of this Lemma a few comments are needed.
The case $p=2$ is reduced after using  the Fourier transform to the following inequality 
\begin{equation}\label{pt.p=2}
C(\rho)\leq  \eps \Big[n^2 \Big(\widehat\rho(0)-\widehat \rho(\frac \xi n)\Big)+1\Big]+\frac{1}{\eps(1+|\xi|^2)}, \quad \forall \, \xi\in \rr^d. 
\end{equation}
Using that $\rho$ is a smooth radially symmetric function we obtain that its Fourier transform decays at infinity and moreover, $\widehat\rho(0)-\widehat \rho(\xi)\simeq |\xi|^2$ for $\xi\simeq 0$. This shows the existence of two positive constants $c_1$ and $c_2$ such that
\begin{equation}\label{est.J.1}
\frac{ c_1 |\xi|^2}{1+|\xi|^2}\leq \widehat\rho(0)-\what \rho(\xi)\leq  \frac{c_2|\xi|^2}{1+|\xi|^2}, \ \forall \ \xi\in \rr^d.
\end{equation}
This property implies that inequality \eqref{pt.p=2} holds for all $n\gtrsim \eps^{-1/2}$.

The local version of inequality \eqref{est.balance} is the following one
\begin{align}\label{est.eps}
\|u\|_{L^p(\rr^d)}^p&\lesssim  \eps \|u\|_{W^{1,p}(\rr^d)}^p+\eps^{-1} \|u\|_{W^{-1,p}(\rr^d)}^p\\
\nonumber&= \varepsilon \|(I-\Delta)^{1/2}u\|_{L^p(\rr^d)}^p+\eps^{-1} \|(I-\Delta)^{-1/2}u\|_{L^p(\rr^d)}^p.
\end{align}
We remark that when $p\neq 2$ this inequality is not a consequence of a duality argument since the dual of $W^{1,p}(\rr^d)$ is $W^{-1,p'}(\rr^d)$.
Inequality \eqref{est.eps} holds by  proving that, depending on the Fourier localization of $u$, 
its  $L^p$-norm is controlled by one of the two terms in the right hand side of \eqref{est.eps}.
In fact, for any $0<\beta<1$ using classical multiplier arguments (see  \cite[Ch.~5]{MR2445437}) we have
\begin{equation}\label{star}
\|u\|_{L^p(\rr^d)}\lesssim \beta \|(I-\Delta)^{1/2}u\|_{L^p(\rr^d)}, \quad {\rm supp}\, \widehat u\subset \{\xi: |\xi|\gtrsim  \beta ^{-1}\}
\end{equation}
and
\begin{equation}\label{star2}
\|u\|_{L^p(\rr^d)}\lesssim \beta^{-1} \|(I-\Delta)^{-1/2}u\|_{L^p(\rr^d)}, \quad {\rm supp}\, \widehat u\subset \{\xi: |\xi|\lesssim  \beta ^{-1}\}.
\end{equation}


\begin{proof}[Proof of Lemma \ref{balance-p}]
Let us first make a change of variable to avoid the presence of $\rho_n(x)=n^d\rho(nx)$. Estimate \eqref{est.balance} is equivalent to the following one
\begin{align}\label{est.fara.n}
  C\|u\|_{L^p(\rr)}^p \leq \eps  \Big[n^p & \int_{\rr^d}\int _{\rr^d} \rho(x-y)|u(x)-u(y)|^pdxdy+ \|u\|_{L^p(\rr^d)}^p\Big] \\
\nonumber&+ \eps^{-1} \|(I-n^2\Delta)^{-1/2}u\|_{L^p(\rr^d)}^p.
\end{align}
 We  use  a decomposition of $u$ that has already been  used in \cite{MR2542582}. Let us choose $\eta\in C_c^\infty(\rr^d)$ with 
 $$\int _{\rr^d}\eta=1 \quad \text{and} \quad |\nabla \eta|+|\eta| \lesssim \rho.$$ 
 This choice of $\eta$ can be always done if $\rho$ is positive in some open set.
 We write 
 $$u=v+w,\quad v=\eta\ast u, \quad w=u-v.$$
We now emphasize some important properties of $v$ and $w$.
First of all observe that both of them have the $L^p$-norm controlled by the $L^p$-norm of $u$:
\begin{equation}\label{norms}
\|v\|_{L^p(\rr^d)}\leq C(\eta)\|u\|_{L^p(\rr^d)}, \quad \|w\|_{L^p(\rr^d)}\leq C(\eta)\|u\|_{L^p(\rr^d)}
\end{equation}
and moreover
\[
\|u\|_{L^p(\rr^d)}\leq \|v\|_{L^p(\rr^d)}+\|w\|_{L^p(\rr^d)}.
\]
Since the mass of $\eta$ is one we have the following representation for $w$:
\[
w(x)=\int _{\rr^d}\eta(x-y)(u(x)-u(y))dy.
\]
 H\"older's inequality gives us that 
\begin{align}\label{norm.w}
\int _{\rr^d}|w|^p &\leq \Big(\int _{\rr^d}|\eta|\Big)^{p/p'} \int _{\rr^d}\int_{\rr^d} | \eta(x-y)| |u(x)-u(y)|^pdxdy \\
\nonumber &\leq C(\eta,\rho)\int _{\rr^d}\int_{\rr^d}\rho(x-y) |u(x)-u(y)|^pdxdy.
\end{align}
In the case of $v$, since $\int_{\rr^d} \partial_{x_j}\eta=0$, $j=1,\dots,d$,   we write its gradient  as
$$(\nabla v)(x)=(\nabla \eta \ast u)(x)=\int _{\rr^d} \nabla \eta (x-y)(u(x)-u(y))dy.$$
Thus the same argument as before gives us that 
\begin{align}\label{norm.v}
\int _{\rr^d}|\nabla v|^p &\leq C(\eta) \int _{\rr^d}\int_{\rr^d} |\nabla \eta(x-y)| |u(x)-u(y)|^pdxdy\\
\nonumber &\leq C(\eta,\rho)\int _{\rr^d}\int_{\rr^d}\rho(x-y) |u(x)-u(y)|^pdxdy.
\end{align}

We now prove estimate \eqref{est.fara.n}. In view of \eqref{norm.w}  for $\eps n^p \gtrsim 1$ we have that
\begin{equation}\label{norm.w.1}
\int _{\rr^d}|w|^p \lesssim \eps n^p \int_{\rr^d} \int_{\rr^d}\rho(x-y) |u(x)-u(y)|^pdxdy.
\end{equation}
We claim that $v$ satisfies the following inequality for all $\eps\in (0,1)$ and for all $n\geq 1$
\begin{align}\label{local.norm.v}
\|v\|_{L^p(\rr^d)}\lesssim \eps^{1/p} \|(I-n^2\Delta )^{1/2} v\|_{L^p(\rr^d)}+\eps ^{-1/p}\|(I-n^2\Delta )^{-1/2} v\|_{L^p(\rr^d)}.
\end{align}
Estimates  \eqref{local.norm.v}, \eqref{norms} and  \eqref{norm.v}   imply that
\begin{align*}
\|v &\|_{L^p(\rr^d)}^p\lesssim \eps \Big[ \int_{\rr^d} |v|^p + n^p \int _{\rr^d} |\nabla v|^p\Big]+\eps ^{-1}\|(I-n^2\Delta )^{-1/2} v\|_{L^p(\rr^d)}^p\\
&\lesssim \eps \Big[ \int_{\rr^d} |u|^p + n^p \int _{\rr^d}\int_{\rr^d}\rho(x-y) |u(x)-u(y)|^pdxdy\Big]+\eps ^{-1}\|(I-n^2\Delta )^{-1/2} (\eta \ast u)\|_{L^p(\rr^d)}^p\\
 &\lesssim \eps \Big[ \int_{\rr^d} |u|^p + n^p \int _{\rr^d}\int_{\rr^d}\rho(x-y) |u(x)-u(y)|^pdxdy\Big]+\eps ^{-1}\|(I-n^2\Delta )^{-1/2}  u\|^p_{L^p(\rr^d)}.
\end{align*}
Taking into account the above estimate and estimate \eqref{norm.w.1} for $w$, we obtain 
 that \eqref{est.fara.n} holds.
It remains to prove that \eqref{local.norm.v} holds.
Writing explicitly the terms in the right hand side of \eqref{local.norm.v} we reduce it to the case $n=1$. In this case inequality \eqref{local.norm.v} follows from estimates \eqref{star} and \eqref{star2}.
\end{proof}

\begin{lemma}\label{localize-minus}
Let $\Omega$ be a smooth bounded domain of $\rr^d$ and $p \in (1,\infty)$. For any smooth function $\chi$ supported in  $\Omega$ there exists a positive constant $C(\chi)$ such that
\begin{equation}\label{est.-1}
\|\chi u\|_{W^{-1,p}(\rr^d)}\leq C(\chi) \|u\|_{W^{-1,p}(\Omega)}.
\end{equation}
\end{lemma}

\begin{proof}We consider the case of the smooth function $u$. The general case follows by density.
By the definition of the space $W^{-1,p}(\rr^d)$ there exists a sequence $\varphi_n\in W^{1,p'}(\rr^d)$ with $\|\varphi_n\|_{W^{1,p'}(\rr^d)}\leq 1$ such that
\[
<\chi u,\varphi_n> _{W^{-1,p}(\rr^d),W^{1,p'}(\rr^d)}=\int _{\rr^d} \chi u \varphi_n \rightarrow \|\chi u\|_{W^{-1,p}(\rr^d)}.
\]
Since $\chi$ has the support included in $\Omega$, we have $\chi \varphi_n\in W^{1,p'}_0(\Omega)$ and
\[
\|\chi \varphi_n\|_{W^{1,p'}_0(\Omega)}\leq \|\chi\|_{W^{1,\infty}(\Omega)} \|\varphi_n\|_{W^{1,p'}(\rr^d)}\leq C(\chi).
\] 
Hence
\[
\int _{\rr^d} \chi u \varphi_n\leq \|u\|_{W^{-1,p}(\Omega)}\|\chi \varphi_n\|_{W^{1,p'}_0(\Omega)}
\leq C(\chi) \|u\|_{W^{-1,p}(\Omega)}.
\]
Letting   $n\rightarrow \infty$ we obtain the desired estimate.
\end{proof}

\begin{lemma}\label{local.sup}Let $\rho:\rr^d\rightarrow\rr$ be a radial function with compact support, $\rho(0)\neq 0$,
 $\Omega$ be a domain in $\rr^d$ and $1<p<\infty$. For any smooth function $\chi$  supported in 
$\Omega$ there exists a positive constant $C=C(\chi,p,\Omega)$ such that 
 the following inequality
\begin{align}\label{est.sup}
C n^p & \int _{\rr^d}\int _{\rr^d}  \rho_n(x-y)|(\chi u)(x)-(\chi u)(y)|^pdxdy\\
&
\nonumber \leq  n^p \int _{\Omega_{\rho_n}}\int _{\Omega_{\rho_n}} \rho_n(x-y)|u(x)- u(y)|^pdxdy + \Big(\int _{\rr^d}\rho(z)|z|^p\Big)\int _{\Omega_{\rho_n}} |u|^p.
\end{align}
holds for any  $n>0$ and any $u\in L^p(\rr^d)$.
\end{lemma}

\begin{proof}
Let us first observe that since $\rho$ is radially symmetric and $\rho(0)\neq 0$ we have
$${\rm supp }(\rho_n)=\frac{1}n {\rm supp}(\rho).$$ 
For $x\notin \Omega_\rhon $ and $y\in \Omega$ we have that 
 $\rho_n(x-y)=0$. If $y\notin \Omega$ then $\chi(x)=\chi(y)=0$. Similar things hold if we interchange $x$ and $y$. Hence
\begin{align}\label{iden.1}
n^p  \int _{\rr^d}\int _{\rr^d} \rho_n(x-y)&|(\chi u)(x)-(\chi u)(y)|^pdxdy\\
\nonumber =n^p  &\int _{\Omega_\rhon}\int _{\Omega_\rhon}  \rho_n(x-y)|(\chi u)(x)-(\chi u)(y)|^pdxdy.
\end{align}
Using the following identity
\[
(\chi u)(x)-(\chi u)(y)=\chi(x)(u(x)-u(y))+u(y)(\chi(x)-\chi(y))
\]
we obtain that
\begin{align}\label{iden.2}
n^p  \int _{\Omega_\rhon}\int _{\Omega_\rhon} &  \rho_n(x-y)|(\chi u)(x)-(\chi u)(y)|^pdxdy\\
\nonumber &\lesssim
n^p \|\chi\|_{L^\infty(\Omega)}^p  \int _{\Omega_\rhon}\int _{\Omega_\rhon}  \rho_n(x-y)| u(x)- u(y)|^pdxdy\\
\nonumber &\quad +n^p \int _{\Omega_\rhon}\int _{\Omega_\rhon}  \rho_n(x-y) |u(y)|^p|\chi(x)-\chi(y)|^pdxdy.
\end{align}
Using identity \eqref{taylor-1} for $\chi$ it follows that
\begin{align}
\label{iden.3}
n^p \int _{\Omega_\rhon}& \int _{\Omega_\rhon}  \rho_n(x-y) |u(y)|^p|\chi(x)-\chi(y)|^pdxdy\\
\nonumber &\leq n^p \int _{\Omega_\rhon}\int _{\Omega_\rhon}  \rho_n(x-y) |u(y)|^p|x-y|^p \int _0^1| (\nabla \chi )(y+s(x-y))|^p dsdxdy\\
\nonumber &\leq  \|\chi\|_{W^{1,\infty}(\rr^d)} \int _{\rr^d} \rho(z)|z|^p dz\int _{\Omega_\rhon} |u(y)|^pdy . 
\end{align}
Putting together estimates \eqref{iden.1}, \eqref{iden.2} and \eqref{iden.3} we infer the desired estimate \eqref{est.sup}.
\end{proof}

\begin{proof}[Proof of Lemma \ref{localizare.n}]
From Lemma \ref{balance-p} we know that for any $\eps\in (0,1)$ and $n\eps^{1/p}\gtrsim 1$ the following inequality holds for all $v\in L^p(\rr^d):$ 
$$\|v\|_{L^p(\rr^d)}^p\lesssim \eps n^p \int_{\rr^d}\int _{\rr^d} \rho_n(x-y)|v(x)-v(y)|^pdxdy + \frac {1}\eps \|v\|_{W^{-1,p}(\rr^d)}^p+\eps\|v\|^p_{L^p(\rr^d)}.
$$
We now localize the above inequality by applying it to  $v=\chi u$ where $\chi $ has been extended by zero outside of $\Omega$. Thus
\begin{align*}
\int _\Omega |\chi u|^p&\lesssim \eps n^p \int_{\rr^d}\int _{\rr^d} \rho_n(x-y)|(\chi u)(x)-(\chi u)(y)|^pdxdy + \frac{1}\eps \|\chi u\|_{W^{-1,p}(\rr^d)}^p+\eps \|\chi u\|^p_{L^p(\rr^d)}.
\end{align*}
By Lemma \ref{localize-minus} and Lemma \ref{local.sup}  we deduce  that
\begin{align*}
\int _\Omega |\chi u|^p \lesssim 
 \eps n^p \int_{\Omega_\rhon}\int _{\Omega_\rhon} \rho_n(x-y)|u(x)- u(y)|^pdxdy +\eps\int _{\Omega_\rhon} |u|^p
 +\frac{1}\eps\|u\|_{W^{-1,p}(\Omega)}^p
\end{align*}
and the proof is finished.
\end{proof}

\section{Proof of Theorem \ref{asimp}}\label{proof.main.result}
\setcounter{equation}{0}

Before starting the proof of Theorem \ref{asimp} we need some preliminary results that will be used throughout the proof.

\subsection{Preliminaries}In the following we denote 
$$J_\lambda(x)=\lambda^d J(\lambda x), \, G_\lambda(x)=\lambda^d G(\lambda x),\ \widetilde G(x)=G(-x),\ \widetilde G_\lambda(x)=\lambda^d \widetilde G(\lambda x) .$$
%


\begin{lemma}\label{parts}
The  following \textit{integration by parts}   identities hold
\begin{align}\label{id}
\int _{\rr^d}(J\ast \Phi -\Phi)&(x)\Psi(x)dx=\int _{\rr^d}\Phi(x)(J\ast \Psi-\Psi)(x)dx\\
\nonumber &=-\frac {1}2 \int_{\rr^d}\int_{\rr^d} J(x-y)(\Phi(x)-\Phi(y))(\Psi(x)-\Psi(y))dxdy
\end{align}
and
\begin{equation}\label{idG}
\int _{\rr^d}(G\ast \Phi -\Phi)(x)\Psi(x)dx=\int _{\rr^d}\Phi(x)(\widetilde G\ast \Psi-\Psi)(x)dx.
\end{equation}
\end{lemma}
\begin{proof}
Use Fubini's theorem and in the first case the fact that $J(-z)=J(z)$. 
\end{proof}

\begin{lemma}\label{lemma5}For any $p\in [1,\infty]$
there exist two positive constants $C(p,J)$ and $C(p,G)$ such that 
\begin{equation}\label{est.J}
\|\lambda^2(J_\lambda\ast \psi-\psi)\|_{L^p(\rr^d)}\leq C(p,J)\|D^2\psi\|_{L^p(\rr^d)}
\end{equation}
and 
\begin{equation}\label{est.G}
\|\lambda(\widetilde G_\lambda\ast \psi-\psi)\|_{L^p(\rr^d)}\leq C(p,G)\|\nabla \psi\|_{L^p(\rr^d)}
\end{equation}
hold for all $\lambda>0$ and $\psi\in C^2_c(\rr^d).$
\end{lemma}

\begin{proof}
We treat the cases $p=1$ and $p=\infty$ since the other cases follow by interpolation.
  A Taylor expansion up to the second order gives us that  for any $x,y\in \rr^d$ the following holds  
$$\psi(y)-\psi(x)= \nabla \psi(x) (y-x)+  \int _0^1(1-s) (y-x) D^2 \psi(x+s(y-x))(y-x)^t ds.$$
After a change of variables we have
\begin{align*}\lambda^2 (J_{\lambda}\ast \psi-\psi)(x)&=
\lambda^{d+2}\int_{\rr^d}  J(\lambda(x-y))(\psi(y)-\psi(x))dy=\lambda^2 \int _{\rr^d}J(z)\Big(\psi(x-\frac z\lambda)-\psi(x)\Big)dz\\
&=\lambda^2  \int _{\rr^d}J(z)\Big[ -\frac z\lambda \cdot \nabla \psi(x)+\frac 1{\lambda^2}\int _0^1 (1-s)z D^2\psi (x-\frac{sz}\lambda)z^t ds \Big]dz.
\end{align*}
Since $J$ is radially symmetric we have
\begin{equation}\label{moment.J}
\int _{\rr^d}J(z)z_jdz=0 \quad \text{for all}\ j=1,\dots, d
\end{equation}
and
\begin{equation}\label{second.moment}
\int _{\rr^d}J(z)z_jz_kdz=0 \quad \text{for all}\ 1\leq j\neq k\le  d.
\end{equation}
Those identities give us that
\begin{equation}\label{taylor.j}
\lambda^2 (J_{\lambda}\ast \psi-\psi)(x)= \sum_{j,k=1}^d \int _0^1 (1-s) \int _{\rr^d}J(z)z_j z_k\frac{\partial^2\psi}{\partial x_j \partial x_k} (x-\frac{sz}\lambda)dz\;ds
\end{equation}
and then for $p\in \{1,\infty\}$ inequality \eqref{est.J} holds  with $C(J)=\frac 12\int _{\rr^d}J(z)|z|^2dz$.

In the case of the second estimate \eqref{est.G} we use the  identity:
$$\psi(y)-\psi(x)= \int_0^1 (y-x)\cdot \nabla \psi(x+s(y-x))ds$$
to obtain 
\begin{align}\label{taylor.g}
\lambda(\widetilde G_\lambda &\ast \psi-\psi) (x)= \lambda ^{d+1}\int _{\rr^d}
\tilde G(\lambda (x-y))(\psi (y)-\psi(x))dy\\
&= \lambda \int _{\rr^d} \tilde G(z)\Big(\psi (x-\frac z\lambda)-\psi(x)\Big)dy
\nonumber=\int_{\rr^d}\tilde G(z)\int _0^1 z\cdot \nabla \psi (x-\frac{sz}\lambda)dsdz \\
&=
\nonumber\int_{\rr^d} G(z)\int _0^1 z\cdot \nabla \psi (x+\frac{sz}\lambda)dsdz.
\end{align}
Using the same  arguments as in the first case we obtain the second estimate.
\end{proof}

\medskip

\subsection{Proof of Theorem \ref{asimp}}
We consider the family $\{u_\lambda(t)\}_{\lambda>0}$ defined by
 $$u_\lambda(t,x)=\lambda^d u(\lambda^2 t,\lambda x).$$
  It follows that $u_\lambda$ is a solution of the following
rescaled equation
\begin{equation}\label{rescal}
\left\{
\begin{array}{ll}
(u_\lambda)_t=\lambda^2(J_\lambda\ast u_\lambda-u_\lambda)+\lambda^{d(1-q)+2}(G_\lambda\ast u_\lambda^q -u_\lambda^q),& x\in \rr^d,t>0,\\[10pt]
u_\lambda(0,x)=\varphi_\lambda(x),
\end{array}
\right.
\end{equation}
where $\varphi_\lambda(x)=\lambda^d \varphi(\lambda x)$.

The proof of Theorem \ref{asimp} is divided into four steps. 

%

\medskip

\textbf{Step I. {Estimates on the rescaled solutions $u_\lambda$.}} We recall  \cite[Theorem 1.4]{MR2356418} that  solution $u$ of system \eqref{cd} satisfies for any $p\in [1,\infty)$ and $t>0$ the following  estimate
\begin{equation}\label{normLpU}
\|u(t)\|_{L^p(\rr^d)}\leq  C(p,\|\varphi\|_{L^1(\rr^d)}, \|\varphi\|_{L^\infty(\rr^d)}) (t+1)^{-\frac{d}{2} \left(1-\frac{1}{p}\right)}.
\end{equation}

In the sequel we will denote by $C$ a constant that may change from line to line, may depend  on $\|\varphi\|_{L^1(\rr^d)}$ and $\|\varphi\|_{L^\infty(\rr^d)}$
but it is independent of the scaling parameter $\lambda$. 
In the following lemmas the constant $M$ will  depend on $\|\varphi\|_{L^1(\rr^d)}$ and  $\|\varphi\|_{L^\infty(\rr^d)}$. We will not make explicit this dependence unless this is necessary. 

\begin{lemma} \label{linfty-l2}
For any $0<t_1<t_2<\infty$ and $p\in [1,\infty)$ there exists a positive constant $M=M(t_1,p)$ 
such that
$$\|u_{\lambda}\|_{L^\infty((t_1,t_2),\, L^p(\rr^d))}\leq M$$
holds for all $\lambda>0$.
\end{lemma}
\begin{proof}
Using  estimate \eqref{normLpU} and the fact that the rescaled solutions satisfy
 $$\|u_{\lambda}(t)\|_{L^p(\rr^d)} =\lambda^{d\left(1-\frac{1}{p}\right)}\|u(\lambda^2 t)\|_{L^p(\rr^d)},$$ we deduce that
for any $p\in [1,\infty)$ and $t>0$ the following inequality holds for all $\lambda>0$:
\begin{equation}\label{normLp}
\|u_{\lambda}(t)\|_{L^p(\rr^d)}\leq  C \left(\frac{\lambda^2}{\lambda^2 t+1}\right)^{\frac{d}{2} \left(1-\frac{1}{p}\right)}\leq  C t^{-\frac{d}{2} \left(1-\frac{1}{p}\right)}.
\end{equation}
Using that $t\geq t_1$ we obtain the desired estimate.
\end{proof}

\begin{lemma} \label{est.2.J}
For  any $0<t_1<t_2<\infty$ there exists a positive constant
$M=M(t_1)$  such that the following inequality 
$$\lambda^2 \int_{t_1}^{t_2} \int_{\rr^d}\int_{\rr^d} J_{\lambda} (x-y) (u_{\lambda}(t,x)-u_{\lambda}(t,y))^2\; dxdy dt\leq M$$
holds for all $\lambda>0$.
\end{lemma}
\begin{proof}
Multiplying  (\ref{rescal}) by $u_{\lambda}$ and integrating over $\rr^d$ we get
\begin{align}\label{energy}
\frac 12\frac{d}{dt} \|u_{\lambda}(t)\|^2_{L^2(\rr^d)} &=\int_{\rr^d} \lambda^2 (J_{\lambda}\ast  u_{\lambda}-u_{\lambda}) u_{\lambda}(t)\; dx+\int_{\rr^d} \lambda^{d(1-q)+2} (G_{\lambda}\ast u^q_{\lambda}-u^q_{\lambda})\; u_{\lambda}(t)\; dx.
\end{align}
Using that  $G_\lambda$ has mass one the last term in the above identity is negative. Indeed,
\begin{align*}
\int_{\rr^d}  (G_{\lambda}\ast &u^q_{\lambda})(t,x) u_{\lambda}(t,x)\; dx=\int_{\rr^d} \int _{\rr^d} G_\lambda (x-y)u_\lambda^q(t,y)u_\lambda(t,x)dxdy\\
&\leq \int_{\rr^d} \int _{\rr^d} G_\lambda (x-y) \Big( \frac {q}{q+1} u_\lambda^{q+1}(t,y)+\frac 1{q+1}u_\lambda^{q+1}(t,x)\Big)dxdy=\int _{\rr^d}u_\lambda^{q+1}(t,x)dx.
\end{align*}
Next,  integrating \eqref{energy} over the interval $(t_1, t_2)$ and using identity \eqref{id} we obtain 
$$ \|u_{\lambda}(t_2)\|^2_{L^2(\rr^d)}+{\lambda^2} \int_{t_1}^{t_2} \int_{\rr^d}\int_{\rr^d} J_{\lambda} (x-y) (u_{\lambda}(t,x)-u_{\lambda}(t,y))^2\; dxdydt\leq \|u_{\lambda}(t_1)\|^2_{L^2(\rr^d)}.$$ 
Using inequality (\ref{normLp})  in the  case $p=2$ we conclude that
$${\lambda^2}\int_{t_1}^{t_2} \int_{\rr^d}\int_{\rr^d} J_{\lambda} (x-y) (u_{\lambda}(t,x)-u_{\lambda}(t,y))^2\; dxdydt\leq C t_1^{-\frac{d}{2}}$$
and the proof finishes.
\end{proof}

\begin{lemma}\label{tderivative}
 For any  $0<t_1<t_2<\infty$ there exists a positive constant $M=M(t_1)$ such that
$$\|u_{\lambda,t}\|_{L^2((t_1,t_2),\, H^{-1}(\rr^d))}\leq M$$
holds for all $\lambda>1$.
\end{lemma}
\begin{proof}
Multiplying  (\ref{rescal}) by $\psi\in C^2_c( \rr^d)$, integrating over $\rr^d$ and using Lemma \ref{parts} we get
\begin{align*}
\int_{\rr^d} u_{\lambda,t}(t,x) \psi(x)\; dx&= \int_{\rr^d} \lambda^2 (J_{\lambda}\ast u_{\lambda}-u_{\lambda}) \psi(x)\; dx+\int_{\rr^d} \lambda^{d(1-q)+2} (G_{\lambda}\ast u^q_{\lambda}-u^q_{\lambda})\; \psi(x)\; dx\\
&= -\frac {\lambda^2}2\int_{\rr^d} \int _{\rr^d} J_{\lambda}(x-y )( \psi(x) -\psi(y))(u_{\lambda}(t,x)-u_\lambda(t,y))\; dxdy\\
&\quad +\int_{\rr^d} \lambda^{d(1-q)+2} (\widetilde{G}_{\lambda}\ast \psi-\psi)\;  u^q_{\lambda}(t,x)\; dx,
\end{align*}
where $\widetilde G_{\lambda}(x)=G_{\lambda}(-x)$. 
Using Cauchy's inequality, the fact that $\lambda>1$ and $q\geq 1+1/d$ we get
\begin{align*}
\Big|\int_{\rr^d}  u_{\lambda,t}(t,x) \psi(x)\; dx\Big| &\leq  \Big(\frac {\lambda^2}2\int_{\rr^d} \int _{\rr^d} J_{\lambda}(x-y )( \psi(x) -\psi(y))^2dxdy\Big)^{1/2}\\
&\quad \times \Big(
 \frac {\lambda^2}2\int_{\rr^d} \int _{\rr^d} J_{\lambda}(x-y)(u_{\lambda}(t,x)-u_\lambda(t,y))^2\; dxdy\Big)^{1/2}
\\
& \quad \quad +\|\lambda (\widetilde{G}_{\lambda}\ast \psi-\psi)\|_{L^2(\rr^d)}\Big( \int_{\rr^d} |u_{\lambda}(t,x)|^{2q}\; dx\Big)^{1/2}.
\end{align*}
Applying Lemma \ref{est.ariba} to $J_\lambda$ and $\psi$, Lemma \ref{lemma5} to $\tilde G_\lambda$  and estimate \eqref{normLp} to $u_\lambda $ we deduce that
\begin{align}\label{ineq}
\Big|\int_{\rr^d} u_{\lambda,t} &(t,x)\psi(x)dx\Big| \\
\nonumber &\lesssim   \| \psi\|_{H^1(\rr^d)} \Big[
\big( {\lambda^2}\int_{\rr^d} \int _{\rr^d} J_{\lambda}(x-y)(u_{\lambda}(t,x)-u_\lambda(t,y))^2\; dxdy\big)^{1/2} +t^{-\frac d4(2q-1)}\Big]. 
\end{align}
Thus
$$\|u_{\lambda,t}(t)\|_{H^{-1}(\rr^d)}^2\lesssim \lambda^2\int_{\rr^d} \int _{\rr^d} J_{\lambda}(x-y)(u_{\lambda}(t,x)-u_\lambda(t,y))^2\; dxdy+t^{-\frac d2(2q-1)}.$$
Integrating this inequality on the time interval $(t_1,t_2)$ and then applying Lemma  \ref{est.2.J} we obtain the desired result.
\end{proof}

\medskip

\textbf{Step II. Compactness in $L^1_{loc}((0,\infty), L^1(\rr^d))$.} We first establish the compactness
of the family $\{u_\lambda\}$ in $L^1_{loc}((0,\infty)\times \rr^d)$. Using estimates on the tail of $\{u_\lambda\}$ we will obtain  strong convergence in $L^1_{loc}((0,\infty), L^1(\rr^d))$.

 Lemma \ref{linfty-l2} and Lemma \ref{tderivative} give us that $\{u_\lambda\}$ is uniformly bounded in $L^\infty_{loc}((0,\infty), L^2_{loc}(\rr^d))$ and 
$\{\partial _t u_\lambda\}$ is uniformly bounded in $L^2_{loc}((0,\infty), H^{-1}(\Omega))$ for any bounded domain $\Omega$ of $\rr^d$. 
Taking into account that $L^2(\Omega)$ is compactly embedded in $H^{-\eps}(\Omega)$ for any $\eps>0$, and  $H^{-\eps}(\Omega)$ is continuously embedded
in $H^{-1}(\Omega)$ for $0<\eps<1$, by classical compactness arguments (\cite{MR916688}, Corollary 4, p.~85) we deduce that $\{u_\lambda\}$ is relatively compact in 
$C([t_1,t_2], \, H^{-\eps}(\Omega))$ for all $0<t_1<t_2$ and $0<\eps<1$. Extracting a subsequence we get
$$u_{\lambda_n}\rightarrow U\ \text{in}\ C([t_1,t_2], \, H^{-\eps}(\Omega)).$$

Using  estimate \eqref{normLp} we obtain that  for each fixed $t>0$, the family $\{u_{\lambda_n}(t)\}_{n\geq 1}$ is uniformly bounded in $L^p_{loc}(\rr^d)$. Then  any subsequence $\{u_{\lambda_{k_n}}(t)\}_{n\geq 1}$  weakly convergent should converge to $U(t)$. Indeed, if $u_{\lambda_{k_n}}(t) \rightharpoonup v  $ in $L^p(\Omega)$ then $u_{\lambda_{k_n}}(t) \rightharpoonup v  $ in $\mathcal{D}'(\Omega)$ and hence $v=U(t)$. This fact shows that for every $t>0$ and $p\in (1,\infty)$  we have
$$u_{\lambda_n} (t)\rightharpoonup U(t)\quad \text{in }\ L^p_{loc}(\rr^d).$$

The uniform bound given in \eqref{normLp} of $\{u_\lambda(t)\}$ transfers to $U(t)$. 
Hence, the limit 
point $U$  belongs to $L^\infty_{loc}((0,\infty),L^p(\rr^d))$ for all $1< p<\infty$ and moreover
 we get that
\begin{equation}\label{norma.p.U}
\|U(t)\|_{L^p(\rr^d)}\leq \liminf _{\lambda\rightarrow \infty}\|u_\lambda(t) \|_{L^p(\rr^d)}\leq \frac {C}{t^{\frac d2(1-\frac 1p)}}, \quad \forall\ t>0.
\end{equation}

Let us now prove the strong convergence in $L^1_{loc}((0,\infty)\times \rr^d)$.
Lemma \ref{linfty-l2}, Lemma \ref{est.2.J} and Lemma \ref{tderivative} show  that for any $0<t_1<t_2<\infty$ 
there exists  $M=M(t_1,\|\varphi\|_{L^1(\rr^d)},\|\varphi\|_{L^{\infty}(\rr^d)})$ such that 
\begin{equation}\label{cond.1}
\|u_\la\|_{L^2((t_1,t_2)\times \rr^d)}\leq M,
\end{equation}
\begin{equation}\label{cond.2}
\la^2\int _{t_1}^{t_2}\int _{\rr^d}\int _{\rr^d} J_\la (x-y)(u_\la(t,x)-u_\la(t,y))^2dxdydt\leq M
\end{equation}
and
\begin{equation}\label{cond.3}
\|u_{\la,t}\|_{L^2((t_1,t_2),H^{-1}(\rr^d))}\leq M.
\end{equation}

Let us now choose a  function $\rho$ as in Theorem \ref{rossi-time} such that $0\leq \rho\leq J$. In view of hypothesis (H2) this is always possible. Also \eqref{cond.2} holds with function $\rho$ instead of $J$.
We apply Theorem \ref{rossi-time}   to family $\{u_\lambda\}_{\lambda>0}$ and to the time  interval $(t_1,t_2)$. We obtain that there exists a function
$v\in L^2((t_1,t_2),H^1(\rr^d))$ such that, up to a subsequence, 
 $$u_\la\rightarrow v\quad \text{in}\quad L^2((t_1,t_2);L^2_{loc}( \rr^d)).$$ 
 The previous analysis shows that $v=U$. Thus $U\in L^2_{loc}((0,\infty),H^1(\rr^d))$ and, up to a subsequence
  $$u_\la\rightarrow U\quad \text{in}\quad L^1_{loc}((0,\infty)\times  \rr^d).$$ 
  
We now prove that in fact  $u_\lambda$ strongly  converges to $U$ in $L^1_{loc}((0,\infty),\,L^1(\rr^d))$.
Using a standard diagonal argument the compactness in $L^1_{loc}((0,\infty),\,L^1(\rr^d))$  is reduced to the fact that for any $0<t_1<t_2<\infty$ the following holds
\begin{equation}\label{est.fuera}
\int _{t_1}^{t_2} \|u_\lambda (t)\|_{L^1(|x|>R)}dt\rightarrow 0 \quad \text{as}\quad R\rightarrow \infty, \, \text{uniformly in }\, \lambda\geq1. 
\end{equation}
This follows from the  Lemma below  since the initial data $\varphi$ belongs to  $ L^1(\rr^d)$.

\begin{lemma}\label{est.2r}
There exists a constant $C=C(J,G,\|\varphi\|_{L^1(\rr^d)},\|\varphi\|_{L^\infty(\rr^d)})$ such that the following inequality
\begin{equation}\label{int.2r}
\int _{|x|>2R} u_\lambda (t,x)dx\leq \int _{|x|>R}\varphi(x)dx+C(\frac{t}{R^2}+\frac {t^{1/2}}R)
\end{equation}
 holds for any $t>0$ and $R>0$, uniformly on $\lambda\geq 1$.
\end{lemma}

\begin{proof}
Let $\psi\in C^\infty(\rr^d)$ be such  that $0\leq \psi\leq 1$ and  satisfies $\psi(x)\equiv 0$ for $|x|<1$ and $\psi(x)\equiv 1$ for $|x|>2$. We put $\psi_R(x)=\psi(x/R)$.
We multiply equation \eqref{rescal} by $\psi_R$ and integrate by parts to obtain 
\begin{align*}
\int _{\rr^d} u_\lambda (t,x)\psi_R(x)dx-\int_{\rr^d}\varphi_{\lambda}(x) \psi_R(x)dx=&
\lambda^2\int _0^t \int_{\rr^d} u_\lambda (s,x)  (J_\lambda\ast \psi_R -\psi_R)dxds\\
&+ \lambda^{d(1-q)+2}\int _0^t \int _{\rr^d} u_\lambda^q (s,x)(\widetilde G_\lambda\ast  \psi_R-\psi_R)(x)dxds.
\end{align*}

We now use  Lemma \ref{lemma5} with $p=\infty$, the fact that 
$$\|D^2(\psi_R)\|_{L^\infty(\rr^d)}=R^{-2}\|D^2\psi\|_{L^\infty(\rr^d)},\quad 
\|\nabla \psi_R\|_{L^\infty(\rr^d)}= R^{-1}\|\nabla \psi\|_{L^\infty(\rr^d)}$$
and the conservation of the $L^1(\rr^d)$-norm of $u_\lambda$
to find that
\begin{align}\label{ineq.uR}
\int _{\rr^d} u_\lambda (t,x)\psi_R(x)dx \lesssim &\int_{\rr^d}\varphi_{\lambda}(x) \psi_R(x)dx+
R^{-2}\|D^2\psi\|_{L^\infty(\rr^d)}\int _0^t \int_{\rr^d} u_\lambda (s,x) dx ds\\
\nonumber &+ \lambda^{d(1-q)+1}R^{-1}\|\nabla \psi\|_{L^\infty(\rr^d)}\int _0^t \int _{\rr^d} u_\lambda^q (s,x)dxds\\
\nonumber\lesssim &\int_{|x|>R}\varphi_{\lambda}(x)dx+
R^{-2}\|D^2 \psi\|_{L^\infty(\rr^d)}t \|\varphi\|_{L^1(\rr^d)}\\
\nonumber&+ \lambda^{d(1-q)+1}R^{-1}\|\nabla \psi\|_{L^\infty(\rr^d)}\int _0^t \int _{\rr^d} u_\lambda^q (s,x)dxds.
\end{align}
To estimate the last term in the above inequality we use the decay of  $u_\lambda$  given by \eqref{normLp} and obtain that
$$\lambda^{d(1-q)+1}\int _0^t \int _{\rr^d} u_\lambda^q (s,x)dxds\lesssim \lambda^{d(1-q)+1} \int _0^t\frac {\lambda ^{d(q-1)}ds}{(1+\lambda^2 s)^{\frac {d(q-1)}2}}
=\lambda^{-1}\int _0^{t\lambda ^2} \frac{ds}{(1+ s)^{\frac {d(q-1)}2}}.$$
Since
$$ 
\lim _{x\rightarrow \infty} x^{-1}\int _0^{x^2}\frac{ds}{(1+ s)^{\frac {d(q-1)}2}}= \lim _{x\rightarrow \infty} \frac{2x}{(1+ x^2)^{\frac {d(q-1)}2}}=
\left\{
\begin{array}{ll}
0,& q>1+\frac 1d,\\[10pt]
2, & q=1+\frac 1d,
\end{array}
\right.
$$
we find that
\begin{equation}\label{star3}
\lambda^{d(1-q)+1}\int _0^t \int _{\rr^d} u_\lambda^q (s,x)dxds\lesssim Ct^{1/2}.
\end{equation}
Going back to \eqref{ineq.uR}, using that $\lambda>1$ and  $\psi(x)\equiv 1$ for $|x|> 2$ we get 
$$\int _{|x|>2R} u_\lambda (t,x)dx\leq \int _{|x|>\lambda R} \varphi(x)dx+ C(\frac{t}{R^2}+\frac {t^{1/2}}R)
\leq \int _{|x|> R} \varphi(x)dx+ C(\frac{t}{R^2}+\frac {t^{1/2}}R)$$
and the proof of the Lemma is finished.
\end{proof}

In view of \eqref{est.fuera} $\{u_\lambda\}_{\lambda>0}$ is relatively compact in 
$L^1_{loc}((0,\infty),\, L^1(\rr^d))$ so 
$u_\lambda \rightarrow U$ in $L^1_{loc}((0,\infty),\, L^1(\rr^d))$. This result also shows that for a.e. $t>0$ we have
\begin{equation}\label{conv.pointwise.1}
\|u_\lambda(t) -U(t)\|_{L^1(\rr^d)}\rightarrow 0 \quad \text{as} \quad \lambda\rightarrow \infty.
\end{equation}
In particular $\|U(t)\|_{L^1(\rr^d)}=m$, a.e. $t>0$.

This fact will be used in Step IV to obtain the main convergence result of this paper.

\medskip

\textbf{Step IIIa. Passing to the limit.} Using the results obtained  in the previous step we  can pass to the weak limit in the equation involving $u_\lambda$.
Let us choose $0<\tau<t$. For any test function $\psi \in C^\infty_c(\rr^d)$ we multiply equation \eqref{rescal} by $\psi$ and we integrate on $(\tau,t)\times \rr^d$. We get
\begin{align*}
\int_{\rr^d} & u_\lambda(t,x)\psi(x)dx-\int_{\rr^d} u_\lambda(\tau,x)\psi(x)dx\\
=&\int^t_\tau\int_{\rr^d} \lambda^2(J_\lambda\ast u_\lambda-u_\lambda)\psi(x)dxds
+
\lambda^{d(1-q)+2}\int^t_\tau\int_{\rr^d} ( G_\lambda\ast u_\lambda^q -u_\lambda^q)\psi(x)dxds\\
=&\int^t_\tau\int_{\rr^d} \lambda^2(J_\lambda\ast \psi-\psi)u_\lambda(s,x)dxds
+
\lambda ^{d(1-q)+2}\int^t_\tau\int_{\rr^d} (\widetilde G_\lambda\ast \psi-\psi)u^q_\lambda(s,x)dxds.
\end{align*} 
First of all observe that since for any $t>0$,   $u_\lambda(t)\rightharpoonup U(t)$ in $L^p_{loc}(\rr^d)$, $1< p<\infty$, we have
\begin{align*}
\int_{\rr^d}  u_\lambda(t,x)\psi(x)dx-\int_{\rr^d} u_\lambda(\tau, x)\psi(x)dx \rightarrow
\int_{\rr^d}  U(t,x)\psi(x)dx-\int_{\rr^d} U(\tau,x)\psi(x)dx.
\end{align*} 
Using identity \eqref{taylor.j} and   the Lebesgue dominated convergence theorem we obtain that
$$\lambda^2(J_\lambda \ast \psi-\psi)(x) \rightarrow \frac 12\sum _{i,j=1}^n \int _{\rr^d}J(z)z_iz_j \frac{\partial^2\psi}{\partial x_i\partial x_j}=
A\Delta \psi(x),$$
where
 $$A=\frac{1}2\int_{\rr^d} J(z)|z|^2dz.$$
Since $u_\lambda \rightarrow U$ in $L^1((\tau,t)\times \rr^d)$ we obtain by using the
Lebesgue  theorem that
 $$\int^t_\tau\int_{\rr^d} \lambda^2(J_\lambda\ast \psi-\psi)u_\lambda(s,x) dxds\rightarrow A\int^t_\tau\int_{\rr^d} \Delta \psi U(s,x)dxds.$$

For the term involving $\tilde G$ we  prove that
\begin{equation}\label{G}
\lambda ^{d(1-q)+2}\int^t_\tau\int_{\rr^d} (\widetilde G_{\lambda}\ast\psi-\psi)u^q_\lambda(s,x)dxds \rightarrow 
\left\{
\begin{array}{ll}
0,& q>1+\frac 1d,\\[10pt]
\int^t_\tau\int_{\rr^d} B\cdot \nabla \psi (x)U^q(s,x)dxds, & q=1+\frac 1d,
\end{array}
\right.
\end{equation}
where
 $$B=(B_1,\dots, B_d),\  B_j=\int_{\rr^d} G(z)z_jdz .$$
 When $q>1+ 1/d$ we use Lemma \ref{lemma5} for $\widetilde G$ and estimate \eqref{normLp} on the $L^p$-norms of $u_\lambda$  
to get
\begin{align*}
\lambda ^{d(1-q)+2}\int^t_\tau \int_{\rr^d} &|\widetilde G_{\lambda}\ast\psi-\psi)(x)|u^q_\lambda(s,x)dxds \\
& \leq  \lambda ^{d(1-q)+1}\|\lambda (\widetilde G_{\lambda}\ast\psi-\psi)\|_{L^{\infty}(\rr^d)}\int^t_\tau\int_{\rr^d}u^q_\lambda(s,x)dxds\\
&\lesssim  \|\nabla \psi\|_{L^{\infty}(\rr^d)} \; \lambda \int _\tau^t \frac {ds}{(1+\lambda^2 s)^{\frac d2(q-1)}}\\
& \lesssim  \|\nabla \psi\|_{L^{\infty}(\rr^d)}\;  \lambda^{1-d(q-1)}\rightarrow 0 \quad \text{as} \;  \lambda \rightarrow \infty.
\end{align*}

Let us now consider the case $q=1+1/d$.
First observe that identity \eqref{taylor.g}  and Lebesgue convergence theorem give us that for any 
$x\in \rr^d$ 
$$\lambda (\tilde G_\lambda\ast \psi -\psi)(x)\rightarrow B\cdot \nabla \psi (x), \quad
\lambda \rightarrow \infty.$$
The results in the previous step and H\"older's inequality give us that for a.e. $t>0$ the following holds
\begin{align*}
\|u_\lambda ^q(t)-U^q(t)\|_{L^1(\rr^d)}&\lesssim \int _{\rr^d}
|u_\lambda (t)-U(t)|^{1/2} (|u_\lambda(t)|^{q-1/2} +|U(t)|^{q-1/2})dx\\
&\lesssim \|u_\lambda (t)-U(t)\|^{1/2}_{L^1(\rr^d)}(\|u_\lambda(t) \|_{L^{2q-1}(\rr^d)}^{\frac{2q-1}2}+
\|U(t) \|_{L^{2q-1}(\rr^d)}^{\frac{2q-1}2})\\
&\lesssim C(t)\|u_\lambda (t)-U(t)\|_{L^1(\rr^d)}\rightarrow 0, \quad \lambda\rightarrow \infty.
\end{align*}
Hence $u_\la^q\rightarrow U^q$ in $L^1((\tau,t), L^1(\rr^d))$ and in consequence the second convergence in \eqref{G} holds.


 All the above convergences  show that  $U$  satisfies
\begin{align*}
\int_{\rr^d} & U(t,x)\psi(x)dx-\int_{\rr^d} U(\tau,x)\psi(x)dx\\
&=A\int^t_\tau \int_{\rr^d} U(s,x)\Delta \psi(x)dxds+\alpha\int^t_\tau \int_{\rr^d} U^q(s,x)B\cdot \nabla \psi(x)dxds,
\end{align*}
where $\alpha=1$ if $q=1+1/d$ and $\alpha=0$ for $q>1+1/d$.
Thus, when $q>1+1/d$ or $B\neq 0_{1,d}$, $U$ is a distributional  solution of the heat equation
$U_t=A\Delta U$. When  $q=1+1/d$ and $B\neq 0_{1,d}$, $U$ is a distributional solution of 
the equation:
$
U_t=A\Delta U-B\cdot \nabla (U^q).
$
 
 \medskip
 \textbf{Step IIIb. Identification of the initial data.} Let us choose $\tau=0$ in the previous step. 
 Using estimate \eqref{star3} and the mass conservation of $u_\la$
  for any $\psi\in C_c^2(\rr^d)$ we get
 \begin{align*}
\Big|\int_{\rr^d} & u_\lambda(t,x)\psi(x)dx-\int_{\rr^d} u_\lambda(0,x)\psi(x)dx\Big|\\
&\leq \|D^2\psi\|_{L^\infty(\rr^d)}\int _0^t \int _{\rr^d} u_\lambda (s,x)dxds +\lambda ^{d(1-q)+1}\|D\psi\|_{L^\infty(\rr^d)} \int _0^t \int _{\rr^d} u_\lambda^q(s,x)dxds\\
&\lesssim t  \|D^2\psi\|_{L^\infty(\rr^d)}+ t^{1/2} \|D\psi\|_{L^\infty(\rr^d)} .
\end{align*}
Letting $\lambda\rightarrow \infty$ and using that for any $t>0$, $u_\lambda (t)\rightharpoonup U(t)$ in 
$L^2(\rr^d)$
we get
$$\Big|\int_{\rr^d}  U(t,x)\psi(x)dx-m\psi(0)\Big|\lesssim   t  \|D^2\psi\|_{L^\infty(\rr^d)}+ t^{1/2} \|D\psi\|_{L^\infty(\rr^d)}, \forall\, t>0,$$
where $m$ is the mass of the initial data $\varphi$.
This shows that for any $\psi\in C_c^2(\rr^d)$
$$\lim _{t \downarrow 0}\int_{\rr^d}  U(t,x)\psi(x)dx=m\psi(0).$$
We want to show that this is true for any smooth bounded function $\psi$ and then $U(t)\rightarrow m\delta _0$ in the weak sense of nonnegative measures in $\rr^d$.

Let us now choose  $\psi$ a bounded smooth function. For any $\eps>0$ we choose $\psi_\eps\in C_c^2(\rr^d)$ such that $\|\psi-\psi_\eps\|_{L^\infty(\rr^d)}\leq \eps$. Then
\begin{align*}
\Big|\int _{\rr^d} & U(t,x)\psi(x)\; dx-m\psi(0)\Big| \\
&\leq \Big|\int _{\rr^d}  U(t,x)(\psi(x)-\psi_\eps(x))dx\Big|+m|\psi(0)-\psi_\eps(0)|+\Big | \int_{\rr^d} U(t,x)\psi_\eps(x)dx-m\psi_\eps(0)\Big|\\
&\lesssim 2\eps m+\|\psi_\eps\|_{W^{2,\infty}(\rr^d)}(t+t^{1/2}).
\end{align*}
Thus there exists $t_0=t_0(\eps)$ such that for all $t\in (0,t_0)$ the following holds
$$\Big|\int _{\rr^d}  U(t,x)\psi(x)\; dx-m\psi(0)\Big| \leq 4\eps m.$$
This shows that $U(t)$ goes to $m\delta_0$ as $t\rightarrow 0$ in the sense of measures.

In conclusion the limit point $U$ satisfies
$U\in L^\infty_{loc}((0,\infty), L^1(\rr^d))\cap L^2_{loc}((0,\infty), \, H^1(\rr^d))$.
When $q>1+1/d$ or $B= 0_{1,d}$, $U$
is a solution of the heat equation $U_t=A\Delta U$ with $m\delta _0$ initial data. When  $q=1+1/d$ and $B\neq 0_{1,d}$, $U$ is a solution of 
the equation:
\begin{equation}\label{ubar}
\left\{
\begin{array}{ll}
U_t=A\Delta U-B\cdot \nabla U^q,& x\in \rr^d,t>0,\\[10pt]
U(0)=m\delta_0.
\end{array}
\right.
\end{equation} 
Since for any $\tau>0$ we have $U(\tau)\in L^1(\rr^d)$ classical results on parabolic equations show that for any $\tau>0$
$$U\in C((\tau,\infty), L^1(\rr^d))\cap L^\infty((\tau,\infty) \times \rr^d).$$
Using the fact that  the heat system as well as system \eqref{ubar} have a unique solution (see \cite{MR1233647}, \cite{MR1266100} for complete details)
then  the full  sequence $\{u_\lambda\}$, not only a subsequence, converges to $U$.

\medskip 
\textbf{Step IV. The asymptotic behavior.}
We recall that from Step II we have
$$\|u_\lambda(1)-U(1)\|_{L^1(\rr^d)}\rightarrow 0 \quad \text{as} \ \lambda\rightarrow \infty.$$
After setting $t=\lambda^2$ and using the self-similar form of $U(t,x)$
$$U(t,x)=t^{-d/2}U(1,xt^{-1/2})=t^{-d/2}f_m(x t^{-1/2})$$ we obtain that solution $u$ of system 
\eqref{cd} satisfies
$$\lim _{t\rightarrow \infty} \|u(t)-U(t)\|_{L^1(\rr^d)}=0.$$
This is exactly \eqref{lim.t} in the case $p=1$. The general case, $1\leq p<\infty$, follows immediately since
\begin{align*}
\|u(t)-U(t)\|_{L^p(\rr^d)}& \lesssim \|u(t)-U(t)\|_{L^1(\rr^d)}^{\frac 1{2p-1}}\Big( \|u(t)\|_{L^{2p}(\rr^d)} +\|U(t)\|_{L^{2p}(\rr^d)}  \Big )^{\frac{2p-2}{2p-1}}\\
&\leq \|u(t)-U(t)\|_{L^1(\rr^d)}^{\frac 1{2p-1}} t^{-\frac d2 (1-\frac 1p)}=o(t^{-\frac d2 (1-\frac 1p)}).
\end{align*}
The proof of Theorem \ref{asimp} is now completed.


\medskip
 {\bf
Acknowledgements.}

L. Ignat was partially supported by Grant PN-II-ID-PCE-2011-3-0075 of the Romanian National Authority for Scientific Research, CNCS -- UEFISCDI.

T. Ignat was  supported by Grant PN-II-ID-PCE-2012-4-0021 of the Romanian National Authority for Scientific Research, CNCS -- UEFISCDI and by a doctoral fellowship offered by IMAR.   

D. Stancu-Dumitru was  partially supported by grant CNCS-UEFISCDI PN-II-ID-PCE-2011-3-0075 {\it Analysis, Control and Numerical Approximations of Partial Differential Equations}. 

Parts of this paper have been done during the visit of the authors at BCAM-Basque Center for Applied Mathematics, Bilbao, Spain. The authors thank the center for hospitality and support. The authors thank J.L. Vazquez, J.D. Rossi and M. Escobedo for fruitful discussions.

\bibliographystyle{plain}
\bibliography{biblio}
\end{document}